\documentclass[11pt,a4paper]{article}
\usepackage{amsmath,amsthm,amssymb,amsfonts,color,graphics,graphicx}

\newcommand{\C}{{\mathbb{C}}}          
\newcommand{\R}{{\mathbb{R}}}          
\newcommand{\Z}{{\mathbb{Z}}}          

\newcommand{\rr}{\rightarrow}
\newcommand{\lrr}{\longrightarrow}

\newcommand{\calR}{{{\cal R}^\xi}}             %

\newcommand{\na}{{\nabla}}
\newcommand{\nag}{{\nabla^g}}

\newcommand{\grad}{{\mathrm{grad}}\,}
\newcommand{\dx}{{\mathrm{d}}}

\newcommand{\inv}[1]{{#1}^{-1}}

\newcommand{\cinf}[1]{{\mathrm{C}}^\infty_{#1}}

\newcommand{\Scal}{{\mathrm{Scal}}}
\newcommand{\expo}{{\mathrm{e}}}

\newtheorem{teo}{Theorem}[section]

\newtheorem{coro}{Corollary}[section]
\newtheorem{prop}{Proposition}[section]

\newenvironment{Rema}[1][Remark.]{\begin{trivlist}
\item[\hskip \labelsep {\bfseries #1}]}{\end{trivlist}}

\newenvironment{meuenumerate}
{\begin{enumerate}
  \setlength{\itemsep}{2.5pt}
  \setlength{\parskip}{-1pt}
  \setlength{\parsep}{-1pt}}
{\end{enumerate}}

\def\cyclic{\mathop{\kern0.9ex{{+}
\kern-2.2ex\raise-.28ex\hbox{\Large\hbox{$\circlearrowright$}}}}\limits}

\title{Homotheties and topology of tangent sphere bundles}

\author{R. Albuquerque\footnote{Departamento de Matem\'atica da Universidade de \'Evora and Centro de Investiga\c c\~ao em Matem\'atica e Aplica\c c\~oes (CIMA), Rua Rom\~ao Ramalho, 59, 671-7000 \'Evora, Portugal.}\\ 
\vspace{2mm}
rpa@uevora.pt}

\begin{document}


\maketitle


\markright{\sl\hfill  Albuquerque \hfill}

\begin{abstract}

We prove a Theorem on homotheties between two given tangent sphere bundles
$S_rM$ of a Riemannian manifold $M,g$ of $\dim\geq3$, assuming different variable radius
functions $r$ and weighted Sasaki metrics induced by the conformal class
of $g$. New examples are shown of manifolds with constant positive or with constant negative scalar curvature which are not Einstein. Recalling results on the associated almost complex structure $I^G$ and symplectic structure $\omega^G$ on the manifold $TM$, generalizing the well-known structure of
Sasaki by admitting weights and connections with torsion, we compute the Chern
and the Stiefel-Whitney characteristic classes of the \textit{manifolds} $TM$ and $S_rM$.

\end{abstract}

\vspace*{4mm}


\vspace*{2mm}


\vspace*{10mm}

The author acknowledges the support of Funda\c{c}\~{a}o Ci\^{e}ncia e Tecnologia, Portugal,  Centro de Investiga\c c\~ao em Matem\'atica e Aplica\c c\~oes da Universidade de \'Evora (CIMA-UE) and the sabbatical grant SFRH/BSAB/895/2009.



\vspace*{10mm}




\section{Introduction}

This article consists of a study of the main properties which identify the
tangent sphere bundles $S_rM=\{u\in TM:\ \|u\|=r\}$ of a Riemannian manifold
$(M,g)$ with variable radius $r$ and induced weighted Sasaki metric
$g^{f_1,f_2}=f_1\pi^*g\oplus f_2\pi^*g$, where $f_1,f_2$ are $\R^+$-valued
functions on $M$ and $\pi:TM\rr M$ is the bundle map. Recall the well-known
Sasaki metric on $TM$  is just $g^S=g^{1,1}$ induced by the Levi-Civita
connection splitting of $TTM$. Our main results are as follows. 

We consider a conformal change $\lambda g$ by some function $\lambda$ on $M$, then take
both Levi-Civita connections of $g$ and $\lambda g$ and consider, accordingly, 
the lifts of these metrics to $TM$. We obtain very different weighted Sasaki
metrics on $TM$ and induced metrics on the sphere bundles, since the
horizontal subspaces are very different when $\lambda$ is non-constant. So one wishes
to compare the $S_rM$, with radius functions $r,s:M\rr\R^+$ and within the same
conformal class of $M$, through the map $u\stackrel{h}\mapsto\frac{s}{r\sqrt{\lambda}}u$.
For $M$ connected and of dimension $\geq3$ we prove:
\begin{equation}
 (S_rM,g^{f_1,f_2})\ \ \ \mbox{is homothetic via\ }h\ \mbox{to}\ \ \ (S_sM,(\lambda g)^{f'_1,f'_2})
\end{equation}
if and only if $\frac{f'_1}{f_1}\lambda=\frac{s^2}{r^2}\frac{f'_2}{f_2}$, the function
$\lambda$ is constant and one of the following conditions holds: (i) $s/r$ is constant or
(ii) $rs$ is constant. 

Equation (ii) is quite interesting, and reassuring if the reader suspects it is
true. As a corollary it says that, for any positive function $r$ on $M$,
$(S_rM,g^S)$ is isometric to $(S_{\frac{1}{r}}M,g^{1,r^4})$.

We give some applications in the treatment of the $S_sM_R$ of the space-form $M_R$, the locus of $x_1^2+\cdots+x_m^2\pm x_{m+1}^2=R^2$, which has constant sectional curvature $\pm1/R^2$. Using \cite{Alb4}, we prove in Theorem \ref{teoremasobrecurvaturaescalar} that, for $\dim M_R=m\geq3$, no matter the sign $\pm$ or the constants $R,f_1,f_2>0$, we can always find a radius $s>0$ suficiently small such that $S_sM_R$ has constant positive scalar curvature or suficiently large such that the same space has constant negative scalar curvature. These are examples of manifolds with constant $\Scal$ but which are not Einstein.

Proceeding with the weighted metric $G=g^{f_1,f_2}$ on $TM$, we define a compatible almost Hermitian
structure $(G,I^G,\omega^G)$, which is a generalization of the canonical or Sasaki almost Hermitian
structure on $TM$. In our case we also allow $\na$ to have torsion. Then the integrability equations
of $I^G$ and $\omega^G$ reserve distinguished roles for the functions $f_1/f_2$ and $f_1f_2$
respectively, both implying the torsion to be of certain so-called vectorial type. In principle
having no relation, notice the similarity of these equations with the two cases (i) and
(ii) above! Finally, the two functions only have to be both constant, the curvature of
$\na$ flat and the torsion zero if and only if we require the defined structure on $TM$ to
be K\"ahler.

We also determine the characteristic classes of the \textit{manifold} $TM$. The Chern classes of
$(TM,I^G)$ are proved to agree with the Pontryagin classes of $M$. Moreover, they do not
depend on the metric connection $\na$. The Stiefel-Whitney characteristic classes of $S_rM$ are also
found. In particular we conclude that any tangent sphere bundle of an oriented manifold is a spin manifold. 

The motivation for the present article is the discovery of a natural $\mathrm{G}_2$-structure
on $S_1M$, for any $M$ oriented of dimension 4, which is having many 
developments and good expectations, cf. \cite{Alb2,Alb2.1}. However, here we just
complete an independent study of the $S_rM$ initiated in \cite{Alb3.0,Alb4}.

Parts of this article were written during a sabbatical leave at Philipps Universit\"at
Marburg. The author wishes to thank the hospitality of the Mathematics Department of
Philipps Universit\"at and specially expresses his gratitude to Ilka Agricola.


\section{Riemannian geometry of the tangent bundle}

\subsection{The tangent bundle}

Let $M$ be an $m$-dimensional smooth manifold without boundary. Let $\pi:TM\rr
M$ be the tangent bundle so that $\pi(u)=x,\ \forall u\in T_xM,\ x\in M$. Then
$V=\ker\dx\pi$ is known as the vertical bundle tangent to $TM$. There is a
canonical identification $V=\pi^*TM$ and an exact sequence over the
manifold $TM$:
\begin{equation}\label{sequenciaexata}
 0\lrr V\lrr TTM\stackrel{\dx\pi}\lrr\pi^*TM\lrr0 .
\end{equation}
The tangent bundle $TM$ is endowed with a natural vertical vector field, denoted
$\xi$, which is succinctly defined by $\xi_u=u$. 

Let $\na$ be a connection on $M$. Then there is a complement for $V$
\begin{equation}
 H=\{X\in TTM:\ \pi^*\na_X\xi=0\}.
\end{equation}
Indeed $H$ is $m$-dimensional and $\pi^*\na_\cdot\xi$ is the vertical projection onto $V$. For any vector field $X$ over $TM$ we may always find the unique decomposition ($\na^*$ denotes the pull-back connection)
\begin{equation}
 X=X^h+X^v=X^h+\na^*_X\xi.
\end{equation}
Now, $\dx\pi$ induces a vector bundle isomorphism between $H$ and $\pi^*TM$, by (\ref{sequenciaexata}), and we have $V=\pi^*TM$. Hence we may define an endomorphism 
\begin{equation}\label{aplicacaotheta}
 B:TTM\lrr TTM
\end{equation}
sending $X^h$ to the respective $B X^h\in V$ and sending $V$ to 0. We also define an
endomorphism, \textit{denoted} $B^{\mathrm{ad}}$, which gives $B^{\mathrm{ad}}X^v\in
H$ and which
annihilates $H$. In particular $B^{\mathrm{ad}}B X^h=X^h$ and $B^2=0$. Sometimes
we call $B X^h$ the mirror image of $X^h$ in $V$. The map $B$ appears also
in \cite{Alb2}. We endow $TTM$ with the direct sum connection $\na^*\oplus\na^*$,
which we sometimes denote by $\na^*$. We have in particular that
$\na^*B=\na^*B^{\mathrm{ad}}=0$.

Notice the canonical section $\xi$ can be mirrored by $B^{\mathrm{ad}}$ to give a
horizontal
canonical vector field $B^{\mathrm{ad}}\xi$. In the torsion free case, the latter is
known
as the spray of the connection, cf. \cite{Dom,Sakai}, or the geodesic field, cf.
\cite{Geig}. It has the further property that $\dx\pi_u(B^{\mathrm{ad}}\xi)=u,\
\forall u\in
TM$. Away from the zero section, we have a line bundle $\R\xi\subset V$ and therefore
a line sub-bundle too of $H$.

\subsection{Natural metrics} 

Suppose the previous manifold $M$ is furnished with a Riemannian metric $g$ and a
linear connection. We also use $\langle\ ,\ \rangle$ in place of the symmetric tensor
$g$; this same remark on notation is valid for the pull-back metric on $\pi^*TM$. We
recall from \cite{Dom,Sasa} the now called Sasaki metric in $TTM=H\oplus V$: it is
given by $g^S=\pi^*g\oplus\pi^*g$ (originally, with the Levi-Civita connection). With
$g^S$, the map $B_|:H\rr V$ is an isometric morphism and $B^{\mathrm{ad}}$
corresponds with
the adjoint of $B$. We stress that $\langle\ ,\ \rangle$ on $TTM$ always refers
to the Sasaki metric.

Let $\varphi_1,\varphi_2$ be any given functions on $M$ and let
\begin{equation}
G=g^{f_1,f_2}=f_1\pi^*g\,\oplus\,f_2\pi^*g
\end{equation}
with
\begin{equation}
 f_1=\expo^{2\varphi_1},\ \ \ f_2=\expo^{2\varphi_2}.
\end{equation}
Obviously, we convention all these functions to be composed with $\pi$ on the right hand side when used on the manifold $TM$.

\begin{Rema}
With the canonical vector field $\xi$ we may produce other symmetric bilinear forms
over $TM$: first the 1-forms $\eta=\xi^\flat$ and
$\theta=\xi^\flat\circ B=(B^{\mathrm{ad}}\xi)^\flat$
and then the three symmetric products of these. Actually one may see that $\theta$
does not depend on a chosen connection which is metric; cf. last remark in section
\ref{comandsympstru}. The classification of all $g$-induced natural metrics on $TM$
may be found e.g. in \cite{Abb1,AbbCalva}.
\end{Rema}

\subsection{Metric connections}

Let us assume from now on that the connection on $M$ is metric, which implies
$\na^*g^S=0$. It
is well-known that $\na^{f_1}=\na+C_1$, with
\begin{equation}\label{TransfLeviCivitaundercg}
C_1(X,Y)=X(\varphi_1)Y+Y(\varphi_1)X-\langle X,Y\rangle\grad\varphi_1,
\end{equation}
is a metric connection for $f_1g$ on $M$, with the same torsion as $\na$ since $C$ is
symmetric.

For any function $\varphi$, recall the usual identities
$X(\varphi)=\dx\varphi(X)=\langle\grad\varphi,X\rangle$, adopted throughout. On $TM$
we shall use the functions $\partial\varphi(u)=\dx\varphi_{\pi(u)}(u),\ \forall u$.
In other words,
\begin{equation}
 \partial\varphi=\langle B\pi^*\grad\varphi,\xi\rangle
\end{equation}
where $B$ is the mirror map (\ref{aplicacaotheta}). And we agree on lifting gradient
vector fields only to $H$.

We have that $\na^{*,f_1}=\na^*+\pi^*C_1$ makes $f_1\pi^*g$ parallel on $H$. On the
vertical side, $\na^{*,f_2}$, defined by
\begin{equation}
\na^{*,f_2}_XY=\na_X^*Y+B\pi^*C_2(X,B^{\mathrm{ad}}Y)
\end{equation}
$\forall X,Y$ vector fields on $TM$, makes $f_2\pi^*g$ parallel. Henceforth,
the connection $\na^{*,f_1}\oplus\na^{*,f_2}$ is metric for $G=g^{f_1,f_2}$.
\begin{prop}\label{torsaodenablaoplusnabla}
i)\, The torsion of $\na^*\oplus\na^*$ is $\pi^*T^\na+\calR$.\\
ii)\, The connection $\na^{*,f_2,'}_XY=\na^*_XY+X(\varphi_2)Y$ is metric on
$(V,f_2\pi^*g)$.
\end{prop}
The proof of this result is immediate. The vertical part in i) is defined via the curvature,
$\calR(X,Y)=\pi^*R^\na(X,Y)\xi$. We remark it is $\na^{*,f_1}$ and the connection in
ii) which enter in the Levi-Civita connection $\na^G$ of $G$. Formulas for the curvature
are well-known, cf. \cite{AbbCalva,Alb4,Dom,KowSek2}.

\subsection{Homotheties of $TM$}
\label{HomoofTM}

Suppose we have a conformal change of the metric $g$ on the base $M$. With
$\lambda=\expo^{2\varphi}$ and $\varphi\in \cinf{M}$ we pass to the metric
\begin{equation}
 g'=\lambda g=\lambda\langle\ ,\ \rangle.
\end{equation}
Let us distinguish by $T'M$ the tangent manifold of $M$ with the metric $g'$, when
necessary. For the rest of the section we restrict to the Levi-Civita connection
\begin{equation}
 \na=\nag.
\end{equation}
Notice $TTM=H\oplus V=H'\oplus V$ and we conform to our previous remarks on notation.

Let also $t:M\rr\R\backslash\{0\}$ be a smooth function. Then we may consider the
isomorphism (letting $\hat{h}=\expo^{-\varphi}t$)
\begin{equation}\label{trasconfgtwoistor}
 h:TM\lrr T'M,\qquad h(u)=\expo^{-\varphi}tu=\hat{h}u=:u' \ .
\end{equation}
We treat all given scalar functions like $\varphi$ or $t$, depending on the context,
as functions composed with $\pi$. This implies, for example,
\begin{equation}\label{derivadas}
 X(\varphi)=\dx\varphi(X)=X^h(\varphi)\ .
\end{equation}
Recall the 1-form $\theta$ on $TM$ given by $\theta(X)=\langle B X,\xi\rangle$.
\begin{prop}\label{derideh}
Let $X$ be any vector field on $TM$ and consider the differential map $h_*:TTM\rr
h^*TT'M$. It satisfies the identities $h_*(X^{v})=\hat{h}X^v$ and, more generally, 
\begin{equation}\label{derideh1}
h_*X=X^{h'}+\hat{h}\bigl(\frac{X(t)}{t}\xi+X^v+\partial\varphi.B X 
-\theta(X)B\grad\varphi\bigr)
\end{equation}
where $B$ refers to the decomposition $H\oplus V$.
\end{prop}
\begin{proof}
We know that $\na'=\na+C$ where $C_XY=\dx\varphi(X)Y+\dx\varphi(Y)X-\langle
X,Y\rangle\grad\varphi$ (here $X,Y$ denote vector fields on $M$ or on $TM$).
Since $\pi\circ h=\pi$, then $(h_*X)^{h'}=\inv{(\dx \pi)}(\dx \pi(X))$ and this is
the same as $X^{h'}$, the $H'$-part of $X$. Writing $\xi'$ for the very same
canonical vector field $\xi$ on $T'M$, so that $h^*\xi'=\xi\circ h=\hat{h}\xi$, and
computing, 
\begin{eqnarray*}
\pi^*\na'_{h_*(X)}\xi' &=& h^*\pi^*(\na+C)_Xh^*\xi'\\
&=& \pi^*\na_X(\hat{h}\xi)+B\pi^*C(X,B^{\mathrm{ad}}(\hat{h}\xi)) \\
&=& \dx\hat{h}(X)\xi+\hat{h}\na^*_X\xi+\hat{h}B\pi^*C(X,B^{\mathrm{ad}}\xi)  \\
&=& -X(\varphi)\hat{h}\xi+\expo^{-\varphi}X(t)\xi+\hat{h}X^v+\hat{h}X(\varphi)\xi+\\
& &\qquad\quad\qquad
+\hat{h}(B^{\mathrm{ad}}\xi)(\varphi).B X-\hat{h}\langle B
X,\xi\rangle B\grad\varphi\\
&=& \hat{h}\bigl(\frac{X(t)}{t}\xi+X^v+\partial\varphi.B X 
-\theta(X)B\grad\varphi\bigr)
\end{eqnarray*}
we find the vertical part.
\end{proof}
\begin{Rema}
Notice any tangent vector $X=X^h+X^v=X^{h'}+X^{v'}$ has two decompositions. We have,
cf. figure 1,
\begin{equation}
 \begin{split}
 X^{v'}\ =\ \na'^*_X\xi\ =\ \na_X\xi+B\pi^*C(X,B^{\mathrm{ad}}\xi) \hspace{3.1cm}\\
=\ X^v+\partial\varphi.B X+X(\varphi)\xi-\theta(X)B\grad\varphi,\\
X^{h'}\ =\ X-X^{v'} \ =\ X^h-\partial\varphi.B
X-X(\varphi)\xi+\theta(X)B\grad\varphi\ .
 \end{split}
\end{equation}
\begin{figure}\label{figuraprojs}
 \begin{center}
\includegraphics{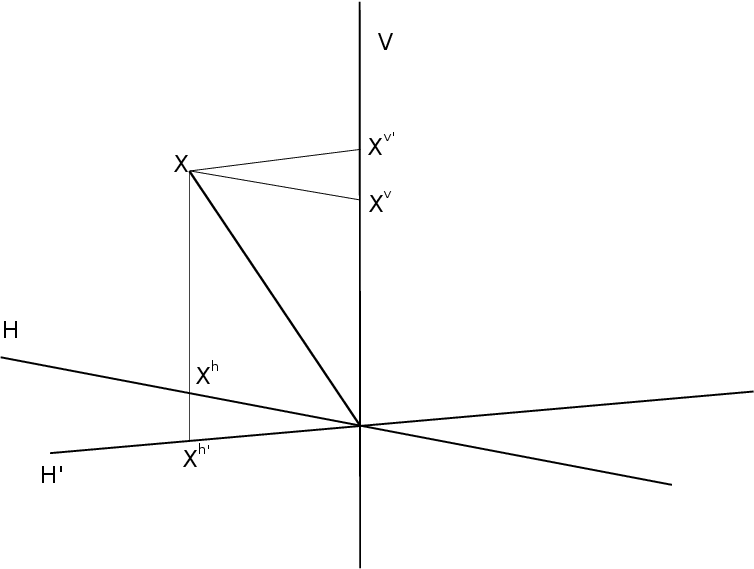}
\caption{The connection induced projections}
\end{center}
\end{figure}
\end{Rema}

Now we suppose $TM$ is endowed with the metric $G=g^{f_1,f_2}$ introduced in previous
sections and we let $T'M$ have the metric $G'=(\lambda g)^{f'_1,f'_2}$ (the four
weight functions are just smooth, positive and defined on $M$).
\begin{teo}\label{homotetiaemTM}
The map $h$ is a homothety (ie. $h^*G'=\psi G$ for some function $\psi$) if and only
if $t$ and $\lambda$ are constants and satisfy
$\frac{f'_1}{f_1}\lambda=t^2\frac{f'_2}{f_2}$. In this case, the latter is the value
of $\psi$.
\end{teo}
\begin{proof}
We write $h_*X=X^{h'}+\hat{h}E(X)$ defining $E$ from (\ref{derideh1}). Then solving
the equation above with vertical vector fields $X_1,X_2$ we immediately find
\[ h^*G'(X_1,X_2)= \psi G(X_1,X_2)\ \,\mbox{if and only if}\,\
\lambda\hat{h}^2f'_2=\psi f_2\ \mbox{i.e.}\ t^2f'_2=\psi f_2 .\]
In particular, $\psi$ is only defined on $M$. Notice we may write
\[  E_{au}(X^h)=aE_u(X^h),\qquad \forall a\in\R,    \]
because $\xi$ is also hidden linearly in $\partial \varphi$ and $\theta$. Picking two
horizontal lifts and having in mind that $t$ and $\psi$ are only defined on $M$, it
is then easy to deduce that a necessary condition for $h$ to be a homothety is that
$E(X^h)=0$ for all $H$-horizontal $X$. Now
\[ t\langle E(B^{\mathrm{ad}}\xi),\xi\rangle=t\langle \frac{\partial
t}{t}\xi+\partial\varphi.\xi-\|\xi\|^2B\grad\varphi ,\xi\rangle=(\partial
t+t\partial\varphi-t\partial\varphi)\|\xi\|^2 \]
and hence $\partial t=\dx t(B^{\mathrm{ad}}\xi)=0$. Choosing any $X$ horizontal and
orthogonal to $B^{\mathrm{ad}}\xi$ (recall $m>1$), we find 
$0=\langle E(X),B X\rangle=\partial\varphi\|X\|^2=0\Leftrightarrow\partial\varphi=0$,
as we wished. In particular, $\na=\na'$. Finally, solving the equation above for
horizontal vector fields $X_1,X_2$ we get $f'_1\lambda=\psi f_1$. For generic vectors
the result follows.
\end{proof}
Generalizing the Theorem for the case of two conformal changes we have: the map $h$
from $(\lambda_1g)^{f_1,f_2}$ to $(\lambda_2g)^{f'_1,f'_2}$ is a homothety if and
only if
\begin{equation}
t,\frac{\lambda_2}{\lambda_1}\:\ \mbox{are constants}\quad\mbox{and}\quad
 \frac{f'_1\lambda_2}{f_1\lambda_1}=t^2\frac{f'_2}{f_2}.
\end{equation}

\subsection{Homotheties of $S_rM$}

Let $r,s\in\cinf{M}(\R^+)$ and recall the tangent sphere bundle of radius $r$
\begin{equation}
 S_rM=\{u\in TM:\ \|u\|_g^2=r^2\}
\end{equation}
submanifold of $TM$, for which we have
\begin{equation}
 S_rM\,=\,S'_1M
\end{equation}
using the metric $\lambda g$ to define $S'_sM$ with
$\lambda=r^{-2}=\expo^{2\varphi}$. Consider the smooth function $N=r^{-2}\|\xi\|^2$ on $TM$, cf. formula (\ref{derivadas}). Then $S_rM=\inv{N}(1)=\{u\in TM:\ G(\xi_u,\xi_u)=1\}$
where $G=g^{f_1,r^{-2}}$ with $f_1$ any positive function. Using Proposition
\ref{torsaodenablaoplusnabla} to differentiate $N=G(\xi,\xi)$, it is easy to deduce
\begin{equation}\label{tangentesaSrM}
 TS_rM=\bigl\{X\in TTM:\ \langle X,\xi\rangle=rX(r)\bigr\}.
\end{equation}
We have to assume $\varphi_2=\varphi=-\log r$. But of course one just applies $\na^*$
to $\|\xi\|^2-r^2=0$ to easily find the same information. Notice $X\in
TS_rM\Leftrightarrow \langle X^v,\xi\rangle=rX^h(r)$.

We shall consider a more general setting: with $r$ and $\varphi$ independent.

Let $\lambda g$ be any conformal change of the given metric,
$\lambda=\expo^{2\varphi}$. Let $s$ be another positive function on $M$ and consider
the map $h$ from Proposition \ref{derideh} with an appropriate chosen $t$. It restricts to a \textit{diffeomorphism}
\begin{equation}\label{justadiffeo}
 h:S_rM\lrr S_s'M,\quad\quad h(u)=\expo^{-\varphi}\frac{s}{r}u=\hat{h}u\ .
\end{equation}
When is $h$ a homothety for the induced metrics? For a start, only the metrics $G,G'$
constructed as in section \ref{HomoofTM} are relevant, i.e. those induced from
$\na=\nag$ the Levi-Civita connection.
\begin{Rema}
Recall the metric on the right hand side arises from $H'\oplus V$. Since
$h_*:T_uS_rM\rr T_{\hat{h}u}S_s'M$, it is true that we have
\[ rX(r)\,=\,\langle X,\xi\rangle\ \ \Leftrightarrow\ \  s\,(h_*X)(s)\,=\,\langle
h_*X,\hat{h}\xi\rangle' .\]
Indeed, we may write $h_*X=X^{h'}+\hat{h}E(X)$ where $E(X)$ is given in
(\ref{derideh1}) as
\begin{equation}
 EX=\frac{X(t)}{t}\xi+X^v+\partial\varphi.B X-\theta(X)B\grad\varphi.
\end{equation}
but now with the function 
\begin{equation}
  t=\frac{s}{r}\ .
\end{equation}

Also, on vertical vector fields the metrics agree up to the scale, so we find
\begin{eqnarray*}
 \langle h_*X,\hat{h}\xi\rangle' &=& \hat{h}^2\langle EX,\xi\rangle'\\
&=& \expo^{2\varphi}\expo^{-2\varphi}t\langle X(t)\xi+tX^v+t\partial\varphi.B
X-t\theta(X)B\grad\varphi,\xi\rangle\\
&=& t\bigl(\frac{X(s)r-sX(r)}{r^2}\|\xi\|^2+t\langle X^v,\xi\rangle+t\partial\varphi.\theta(X)-t\theta(X)\partial\varphi\bigr)\\
&=& \frac{s}{r}\bigl(rX(s)-sX(r)+\frac{s}{r}rX(r)\bigr)\\
&=& sX(s)\ =\ s\,(h_*X)(s)
\end{eqnarray*}
since on $S_rM$ we have $\|\xi\|^2=r^2$.
\end{Rema}

In the next Theorem we prove that each tangent sphere bundle $S_rM$ with metric $G$
induced from that of $TM$ is quite unique, independently of any of the metric
transformations above and up to the straightforward coincidences expressed in the
corollaries. The reader may notice the impossibility of adapting the arguments used
for Theorem \ref{homotetiaemTM}. We also remark  we were not able to prove the cases of $\dim M=1,2$.

We let $\lambda=\expo^{2\varphi}$ and $r,s,f_1,f_2,f'_1,f'_2$ be any positive
functions on $M$.

Until the end of this section we assume $M$ is connected and $\dim M\geq3$.
\begin{teo}\label{homotetiaSM}
Let $S_rM$ have the induced metric $G=g^{f_1,f_2}$ and let $S_s'M$ have the induced
metric $G'=(\lambda g)^{f'_1,f'_2}$. Then the following are equivalent:
\begin{meuenumerate}
 \item $h:S_rM\rr S_s'M$ is a homothety, ie.  $h^*G'=\psi G$ for some function
$\psi$.
 \item $\lambda$ is constant, $\psi$ verifies simultaneously
$\psi=\frac{f'_1}{f_1}\lambda=\frac{s^2}{r^2}\frac{f'_2}{f_2}$ and one of the following hold:\\
(i) $s/r$ is constant\\
(ii) $rs$ is constant.
\end{meuenumerate}
For the case of the identity ($\hat{h}=1$), we have that it is a homothety if and
only if $\lambda=s^2/r^2$ is a constant and $\frac{f'_1}{f_1}=\frac{f'_2}{f_2}$.
\end{teo}
\begin{proof}
First we notice
\begin{eqnarray*}
 G'(h_*X,h_*Y) &=& f'_1\langle X^{h'},Y^{h'}\rangle'+\hat{h}^2f'_2\langle
EX,EY\rangle'\\
&=& f'_1\lambda\langle X^h,Y^h\rangle+\hat{h}^2\lambda f'_2\langle EX,EY\rangle.
\end{eqnarray*}
Now consider the equation $h^*G'(X,Y)=\psi G(X,Y)$. Choose one vector $X=\xi^\perp$
vertical and orthogonal to $\xi$, and a vector $Y=(\grad r)^\perp$ horizontal and
orthogonal to $\grad r$. Then both $X,Y\in TS_rM$. Indeed, $\langle
X,\xi\rangle=0=rX(r)$ and $\langle Y,\xi\rangle=0=r\langle Y,\grad r\rangle=rY(r)$.
Then for two vertical vector fields, like $X$, we immediately get the necessary
condition $\hat{h}^2\lambda f'_2=\psi f_2\ \Leftrightarrow\
\psi=\frac{s^2}{r^2}\frac{f_2'}{f_2}$. For $X,Y$ we have
$EX=X^v$ and $EY=\frac{Y(t)}{t}\xi+\partial\varphi.B Y-\theta(Y)B\grad\varphi$,
hence
\[ 0=\psi G(X,Y)=G'(h_*X,h_*Y)=f'_2\lambda\hat{h}^2\bigl(\partial\varphi\langle X,BY\rangle
 -\theta(Y) \langle X,B\grad\varphi\rangle\bigr)\ . \]
Now we choose a point $u\in S_rM$ orthogonal to $\grad r$. Then we
may take $X=B\grad r$ and $Y=u\in H$. We have $\langle B Y,X\rangle=0$ and
$\theta(Y)=\langle u,u\rangle=r^2$, so our equation yields $\langle
X,B\grad\varphi\rangle=0$. Equivalently, we must have \,$\grad r\perp\grad\varphi$.

Now suppose $\grad r=0$ on all points of $M$, ie. $r$ is constant. Then $H\subset
TS_rM$. Take any non-vanishing $Z_0\in H$. Then we may further\footnote{This last assumption is not plausible in dimension 2 since we want $Z_0\neq0$, hence the hypothesis on the dimension; although here we may assume that $\grad\varphi$ together with $\grad s$ constitute a basis of $H$ and then try to
solve the system of two quadratic equations and 4 unknowns, in the components of $u$ and
$Z_0$ in that basis, given by $Z_0(t)\xi+t\partial\varphi.B Z_0-t\theta(Z_0)B\grad\varphi =0$,
for that is all we need - here, because ahead the dimension hypothesis is required again.} let $Z_0\in H\cap \{\grad s,\grad\varphi\}^\perp$. In fact, 
in dimension $\geq3$, we may find a point $u$ in each fibre of $S_rM$ such that
$(\partial\varphi)_u=0$ and a vector $Z_0\in H_u$ such that $Z_0(s)=Z_0(\varphi)=0$ and $\theta(Z_0)=0$. Then on the chosen point $u$ we get $E(Z_0)=0$ and so
$h_*Z_0=Z_0^{h'}$. Hence our main equation yields the necessary condition
$f'_1\lambda=\psi f_1$. Going back a little, we then consider any point $u$ and any
$Z_0\in H$ perpendicular to $u$, ie. such that $\xi_u\perp B Z_0$. Then we deduce
\[ G'(h_*Z_0,h_*Z_0)=
f'_1\lambda\|Z_0\|^2+f'_2\lambda\hat{h}^2\bigl(\frac{(Z_0(s))^2}{s^2}r^2+
(\partial\varphi)^2\|Z_0\|^2\bigr)\ =\ \psi f_1\|Z_0\|^2  \]
This immediately implies $Z_0(s)=0$, $\partial\varphi=0$. Since $Z_0$ and $u$ may now be
put in general position, we conclude $s$ and $\varphi$ are constant on $M$, a
connected manifold by assumption, and the Theorem follows.

So now we admit $\grad r\neq0$ at some point $x\in M$. Recall \,$\grad
r\perp\grad\varphi$\, and let $\epsilon=\|\grad r\|$ and $\delta=\|\grad\varphi\|$. 

Hence $u_1=\frac{r}{\epsilon}\grad r\in S_rM$. Notice
$\partial\varphi_{u_1}=\dx\varphi(u_1)=0$. Consider the vector $X_0=\grad r$ and
$X=X_0+\epsilon B X_0$. It is tangent to our sphere bundle at $u_1$ since
\[ \langle X,\xi\rangle=\epsilon\frac{r}{\epsilon}\epsilon^2=r\langle
X_0,X_0\rangle=rX(r).\]
And we have that
\begin{eqnarray*}
h_*X=X^{h'}+\hat{h}EX &=& X^{h'}+\hat{h}\bigl(\frac{X(t)}{t}\xi_{u_1}+\epsilon B
X-\theta(X)B\grad\varphi\bigr) \\   &=&
X^{h'}+\hat{h}\bigl(\frac{X(t)}{t}\frac{r}{\epsilon} +\epsilon\bigr)B X_0
-\hat{h}r\epsilon\,B\grad\varphi.
\end{eqnarray*}
Consider also the tangent vector at $u_1$, $Z=B\grad\varphi$. Then $h_*Z=\hat{h}Z$.
And thus $\psi G(X,Z)=\psi  f_2\epsilon\langle B X_0,Z\rangle=0$; on the other
hand
\begin{eqnarray*}
 h^*G'(X,Z) &=& f'_2\lambda\hat{h}^2\langle \bigl(\frac{X(t)}{t}\frac{r}{\epsilon}+
\epsilon\bigr)B X_0 -r\epsilon\,B\grad\varphi,B\grad\varphi\rangle   \\
&=& -f'_2\lambda\hat{h}^2\epsilon r\delta^2.
\end{eqnarray*}
This implies $\delta=0$, ie. $\varphi$ and hence $\lambda=\expo^{2\varphi}$ are constants.

Therefore the map $h$ verifies $h_*X=X^{h}+\hat{h}(\frac{X(t)}{t}\xi+X^v)$, for any vector field $X\in TS_rM$. Now we consider any horizontal vector $X\in\ker\dx r\cap\ker\dx s$, in particular also tangent to $S_rM$ and orthogonal to $\grad t$ (recall $n\geq2$). Then $X(t)=0$ and, just as we had the result $\hat{h}^2\lambda f'_2=\psi f_2$ using vertical vectors, we have the similar result with horizontal: $\lambda f'_1=\psi f_1$.

Next, we use two generic tangent vectors $X,Y\in TS_rM$. It is easy to see the conformality equation $h^*G'=\psi G$ is finally equivalent to
\[ \langle \frac{X(t)}{t}\xi+X^v,\frac{Y(t)}{t}\xi+Y^v\rangle = \langle X^v,Y^v\rangle,  \]
\[ \frac{X(t)Y(t)}{t^2}r^2+\frac{X(t)rY(r)}{t}+\frac{Y(t)rX(r)}{t} = 0  \]
or 
\[ X(t)Y(t)r+X(t)Y(r)t+X(r)Y(t)t\ =\ 0. \]
Notice this last equation only involves the horizontal part of the vectors, so we assume $X,Y$ as such. Now if we take $X$ orthogonal to $\grad t$, ie. satisfying $X(t)=0$, and take $Y=\grad t$, then we find that $X(r)=0$ or that $X$ is also orthogonal to $\grad r$. Henceforth, $\grad t$ and $\grad r$ are proportional, ie. lie on the same line. In other terms,
\[ \dx t= a\dx r \]
for some function $a$ on $M$. Clearly the equation above may be written as
\[  r\dx t\otimes\dx t+t\dx t\otimes\dx r+t\dx t\otimes\dx r=0 . \]
Hence we have $(ra^2+2ta)\dx r\otimes\dx r=0$. Recalling $r$ is not constant, we either have $t$ constant or $ra+2t=0$. We have both
\[ \dx t=-\frac{2t}{r}\dx r=-\frac{2s}{r^2}\dx r \qquad\quad\mbox{and}\qquad\quad
\dx t=\frac{r\dx s-s\dx r}{r^2} . \]
Hence $-2s\dx r=r\dx s-s\dx r\ \Leftrightarrow\ r\dx s+s\dx r=0$, from which we find $sr=$constant.

Finally all conditions are fulfilled for $h$ to be the expected homothety of ratio $\psi$. The identity map case is trivial.
\end{proof}
Let $g^S=g^{1,1}$ denote the induced Sasaki metric on the tangent sphere bundle and recall we are only considering $\dim M \geq3$.
\begin{coro}
The Riemannian manifold $(S_rM,g^S)$ is homothetic to $(S'_sM,(\lambda g)^S)$ via $h$ if
and only if $\psi=\lambda=\frac{s^2}{r^2}$ and this is a constant. In this case, $h$ is the identity and
$s=\sqrt{\lambda}r$; in other words  $S_rM=S'_sM$. In particular, two tangent sphere
bundles both with the induced Sasaki metric are homothetic if and only if they have
exactly the same radius function, ie., they coincide.
\end{coro}
\begin{coro}\label{casosparticularesdeisometrias1}
Other particular cases are as follows: the Riemannian manifold $(S_rM,g^{f_1,f_2})$ is
isometric via $h$ to $(S'_rM,(\lambda g)^{1,f_2})$ if $f_1=\lambda$ is constant. And
$(S_rM,g^{f_1,f_2})$ is isometric to $(S_1'M,(\lambda g)^{1,r^2f_2})$ if $f_1=\lambda$ and
both $r,f_1$ are constant. Moreover, $(S_rM,g^S)$ is isometric to $(S_1M,g^{1,r^2})$ if
$r$ is constant.
\end{coro}
We have used the metric $G=g^{f_1,r^{-2}}$ on $S_rM$. So we study this case separately.
\begin{coro}\label{casosparticularesdeisometrias2}
Let $S_rM$ be given the metric $G=g^{1,\frac{1}{r^2}}$ and let $S_s'M$ be with the metric $G'=(\lambda g)^{f'_1,\frac{1}{s^2}}$. Then the following three conditions are equivalent: 
\begin{meuenumerate}
 \item the map $h:S_rM\rr S_s'M$ is a homothety.
 \item the functions verify: $\psi=f'_1\lambda=1$, $\lambda$ is a constant and $s/r$ or $sr$ is a constant.
 \item the map $h$ is an isometry.
\end{meuenumerate}
In particular, for any $s,r$ positive constants,
$(S_rM,g^{1,r^{-2}})\simeq(S_sM,g^{1,s^{-2}})\simeq(S_1M,g^S)$.
\end{coro}
\begin{proof}
Indeed we have $\psi=f'_1\lambda=\frac{s^2}{r^2}\frac{r^2}{s^2}=1$.
\end{proof}
\begin{coro}\label{isometriaengracada}
 Let $r$ be any function on $M$. Then $(S_rM,g^S)$ is isometric to 
$(S_{\frac{1}{r}}M,g^{1,r^4})$.
\end{coro}
\begin{proof}
This is due to the second particular case found in the Theorem. We are taking $\lambda=1$ and $s=\frac{1}{r}$ and indeed $\psi=\frac{f'_1}{f_1}\lambda=1=\frac{r^4}{r^4}=\frac{s^2}{r^2}\frac{f'_2}{f_2}$. Also notice we have $sr$ constant.
\end{proof}

\subsection{Applications to space-forms}

Formulas for the curvature of tangent sphere bundles of space-forms are
finally studied here. Let $R>0$ and let
\begin{equation}
 M_R=\bigl\{x\in\R^{m+1}:\ x_1^2+\cdots+x_m^2\pm x_{m+1}^2=R^2\bigr\}
\end{equation}
be an  $m$-dimensional space-form with the induced metric $g$ from Euclidean space.
$M_R$ has constant sectional curvature $\pm1/R^2$. If we conformally change the
metric $g\rightsquigarrow\lambda g$ by a constant, then clearly
$\pm\frac{1}{R^2}\rightsquigarrow\pm\frac{1}{\lambda R^2}$.

Having another $R_1>0$, the map $F:M_{R_1}\lrr M_{R}$ defined by
$x\in\R^{m+1}\longmapsto F(x)=\frac{R}{R_1}x$ induces the following isometry through
differentiation. Writing $f_1=f_2=\frac{R^2}{R_1^2}$ and $s=\frac{R_1r}{R}$, we have
\begin{equation}
 F_*:(S_sM_{R_1},g^{f_1,f_2})\lrr(S_rM_R,g^S),\qquad F_*(x,u)=\frac{R}{R_1}(x,u).
\end{equation}
Indeed, $(F_*)^*g^S=g^{f_1,f_2}$. This isometry and corollaries
\ref{casosparticularesdeisometrias1},\ref{casosparticularesdeisometrias2} give us the next quite interesting result.
\begin{prop}
We have the following isometries
\begin{equation}\label{isometriadeSespacosformas}
(S_1M_R,g^S)\simeq(S_{\frac{1}{R}}M_1,g^{R^2,R^2})\simeq
(S'_{\frac{1}{R}}M_1,(R^2g)^{1,R^2})\simeq
(S'_1M_1,(R^2g)^{1,1})=(S_{1}M_1,(R^2g)^S).
\end{equation}
\end{prop}

Now we apply a general formula from \cite[Proposition 1.6]{Alb4} on the scalar curvature $\Scal$ of
$(S_sM,g^{f_1,f_2})$ for any given constants $f_1,f_2$:
\begin{equation}
 \mathrm{Scal}_{(S_sM,g^{f_1,f_2})}=\frac{1}{f_1}\mathrm{Scal}_{(M,g)}
-\frac{f_2}{4f_1^2}\sum_{i,j,k=1}^m(\calR_{ijk})^2+\frac{(n-1)n}{f_2s^2}.
\end{equation}
Since $\sum_{i,j,k=1}^m(\calR_{ijk})^2=\frac{s^2}{R^4}2n$, where $n=m-1$, we have
\begin{equation}\label{curvaturaescalardasspaceforms}
 \mathrm{Scal}_{(S_sM_R,g^{f_1,f_2})}=\pm\frac{n(n+1)}{f_1R^2}
-\frac{f_2}{4f_1^2}\frac{s^2}{R^4}2n+\frac{(n-1)n}{f_2s^2}.
\end{equation}
In particular the scalar curvature of \eqref{isometriadeSespacosformas} is
 $\pm\frac{n(n+1)}{R^2}-\frac{n}{2R^4}+(n-1)n$. We may say it is rewarding to see the same value in \eqref{curvaturaescalardasspaceforms} for any of the forms in \eqref{isometriadeSespacosformas}. Notice we can also write the scalar curvature of $(S_rM_R,g^S)\simeq(S_{\frac{1}{r}}M_R,g^{1,r^4})$.
\begin{teo}\label{teoremasobrecurvaturaescalar}
For any $m\geq3$ and both cases $\pm$, for any scalars $R,f_1,f_2>0$, we can always find a sufficiently small or large radius $s$ in order to have $(S_sM_R,g^{f_1,f_2})$ with, respectively, positive or negative scalar curvature.
\end{teo}
The proof is clear just by looking at $s$ in (\ref{curvaturaescalardasspaceforms}); the result is partially corroborated by two Theorems in \cite{KowSek2}.

We further remark that the formulas in \cite{Alb4} show the Riemannian metrics we are considering are never Einstein, though $\mathrm{Scal}_{(S_sM_R,g^{f_1,f_2})}$ is constant.

\section{Characteristic classes}

We know not of any reference for the fundamental questions solved in this section.
We extend our study to problems of topology of the tangent and tangent sphere bundles. The first stems from the Riemannian structure.

\subsection{Almost Hermitian structure on $TM$}
\label{comandsympstru}

The pair $TM,g^S$ admits a compatible almost complex structure, also attributed
to Sasaki. It was first studied in \cite{Dom,Sasa} and gave origin in
\cite{Tash} to an almost contact structure on the unit tangent sphere bundle $S_1M$.
For $M$ oriented and dimension $4$ we discovered a natural $\mathrm{G}_2$-structure always existing
on $S_1M$ with the very same metric, cf. \cite{Alb2,Alb2.1}.

We continue the study of $TM$ with the metric $G=g^{f_1,f_2}$ where
$f_1=\expo^{2\varphi_1}$ and $f_2=\expo^{2\varphi_2}$. We let $\na$ denote a metric
connection on $M$ with torsion $T^\na$. The almost complex structure of Sasaki may be
now written as the bundle endomorphism $I^S=B^{\mathrm{ad}}-B$. 

Let 
\begin{equation}
 \psi=\varphi_2-\varphi_1,\qquad\quad\ \ \overline{\psi}=\varphi_2+\varphi_1.
\end{equation}
We then define 
\begin{equation}
 I^G=\expo^\psi B^{\mathrm{ad}}-\expo^{-\psi}B.
\end{equation}
It is easy to see the endomorphism $I^G$ is an almost complex structure compatible with the metric
$G$. We consider also the associated non-degenerate 2-form $\omega^G$ defined by
$\omega^G(X,Y)=G(I^GX,Y)$, $\forall X,Y\in TTM$. Since
$f_1\expo^\psi=f_2\expo^{-\psi}=\expo^{\overline{\psi}}$, it follows that
$\omega^G=\expo^{\overline{\psi}}\omega^S$ where $\omega^S$ is the 2-form associated to the Sasaki
structure $g^S$ and $I^S=B^{\mathrm{ad}}-B$. 

The next Theorem is shown for completeness of exposition. For the Cartan classification of torsions of metric connections see \cite{Agri}. Notice the presence \textit{again} of the quotient and product of the weights $f_1,f_2$!
\begin{teo}[\cite{Alb3.0}]
i) The almost complex structure $I^G$ is integrable if and only if $\na$ is flat and has the
vectorial type torsion $T^\na=\dx\psi\wedge 1$.
In particular, if $\na$ is torsion free, then $I^G$ is integrable if and only if $M$ is Riemannian
flat and $f_2/f_1=$constant.\\
ii)  $(TM,\omega^G)$ is a symplectic manifold if and only if $T^\na=\dx\overline{\psi}\wedge1$.
In particular, with $\na$ the Levi-Civita, $\dx\omega^G=0$ if and only if $f_2f_1=$constant.
\end{teo}
We observe that in the strict case of the Sasaki metric we have $T^\na=0$ as
\textit{necessary} condition for both integrability of $I^S$ and $\dx\omega^S=0$. In
the general case, the two equations are distinguished, as they should, by $\psi$ and
$\overline{\psi}$. Clearly we may draw the following conclusion.
\begin{coro}[\cite{Alb3.0}]
The almost Hermitian structure $(TM,G,I^G,\omega^G)$ is K\"ahler if and only if $M$ is a Riemannian flat manifold ($T^\na=0,\ R^\na=0$) and $f_1,f_2$ are constants. In this case, $TM$ is flat.
\end{coro}
The last assertion follows indirectly from Proposition \ref{torsaodenablaoplusnabla}.
\begin{Rema}
Recall $T^*M$ has a natural symplectic structure. It arises as $\dx\lambda$ where
$\lambda$ is the Liouville 1-form (\cite{Geig}): the unique 1-form $\lambda$ on
$T^*M$ such that on a section $\alpha$
\begin{equation}
 \lambda_\alpha=\alpha\circ\pi_*
\end{equation}

When we introduce the metric, the tangent and cotangent (sphere) bundles become isometric bundles.
With a little computation we find that the 1-form $\theta=\xi^\flat\circ
B=(B^{\mathrm{ad}}\xi)^\flat$ corresponds with the Liouville form, so it does not
depend on the connection. Knowing the torsion of
$\na^*\oplus\na^*$ for any metric connection on $M$, it is easy to deduce, cf.
\cite{Alb2}, that for any radius function we have: 
\begin{equation}
 \dx\theta=\omega^S+\theta\circ T^\na.
\end{equation}
The same is to say $\omega^S$ corresponds with the pull-back of the Liouville symplectic 2-form if
and only if $T^\na=0$. Then a Hamiltonian theory of the geodesic flow is manageable.
We also remark that the geodesic vector field in the sense e.g. of \cite{Geig},
i.e. the vector field $B^{\mathrm{ad}}\xi$ in our setting, is just the same as the
geodesic spray in the sense e.g. of \cite{Dom,Sakai}.
\end{Rema}

\subsection{Chern and Stiefel-Whitney classes of $TM$}

Let us continue with the structures $G,I^G$ on the tangent bundle, induced from any metric connection $\na$, and the same notation from above.

By a deformation retract on the fibres of $\pi:TM\rr M$, there is an identification of cohomology spaces $H^*(M)=H^*(TM)$. This is valid for any coefficient ring. In particular $H^i(TM)=0,\ \forall i>m$. Let $w_j$ denote the $j$-th Stiefel-Whitney class of $M$ --- which is a Stiefel-Whitney class of $TM$ as a vector bundle. Let $w=\sum w_j$ denote the total Stiefel-Whitney class.
\begin{teo}
For any manifold $M$ of dimension $m$, the Euler class of the manifold $TM$ vanishes and the total Stiefel-Whitney class is
\begin{equation}\label{totalSWclassofTTM}
 w(TTM)=w^2=\sum_{j=0}^{[m/2]}w_j^2.
\end{equation}
\end{teo}
\begin{proof}
Being a top degree class, the Euler class must vanish. Since $TTM=\pi^*TM\oplus\pi^*TM$, the Whitney product Theorem and the naturality of the characteristic classes (\cite{MilSta}) immediately give $w(TTM)=w(TM)w(TM)=w^2$. Recall the coefficients are in $\Z_2$, hence the second identity in (\ref{totalSWclassofTTM}).
\end{proof}
\begin{teo}
The Chern classes of the {\em{manifold}} $TM$ with almost complex structure $I^G$ are
the Chern classes of the complexified {\em{tangent bundle}}, $TM\otimes_{\R}\C\rr M$.
\end{teo}
\begin{proof}
The complex structure $I^G$ in $TTM$ is equivalent to $I^S$. One complex isomorphism is given by $f:X\mapsto X^h+\expo^\psi X^v$. Indeed, $\forall X\in TTM$,
\begin{equation*}
\begin{split}
  I^S\circ f(X)=(B^{\mathrm{ad}}-B)(X^h+\expo^\psi X^v)\ =\ -B
X^h+\expo^\psi B^{\mathrm{ad}}X^v \qquad\qquad \\  
= \expo^\psi B^{\mathrm{ad}}X^v-\expo^\psi\expo^{-\psi}B X^h\ =\ f\circ I^G(X).
\end{split}
\end{equation*}
By the functorial properties, we just have to compute the Chern classes of $I^S$. (Another argument: the homotopy induced by $t\psi,\ t\in[0,1]$, preserves the Chern classes.)
Now, the Chern classes of an almost complex manifold $(N,J)$ are the Chern classes of the $\C$-vector bundle $T^+N$, the $+i$-eigenbundle of $J$ where $i=\sqrt{-1}$. In our case, 
\[ T^+TM=H^c=\pi^*TM^c  \]
where $c$ denotes complexification, because of the $\C$-isomorphism induced from
$X\in H\mapsto X+iB X\in T^+TM$. Indeed $I^S(X+iB X)=-B X+iB^{\mathrm{ad}}B
X=i(X+iB X)$. Finally, by trivial reasons, we have $c_j(T^+TM)=c_j(TM^c)$.
\end{proof}
We recall the Chern classes $c_{2j}$ define the Pontryagin classes of $M$, cf. \cite{MilSta},
\begin{equation}
 p_j(M)=(-1)^jc_{2j}(TM\otimes\C).
\end{equation}
Moreover, the Chern classes of $(TM,I^G)$ do not depend on the connection $\na$.

\subsection{Stiefel-Whitney classes of $S_rM$}

Now let $m=n+1$ and let $r>0$ be a scalar function on $M$. We continue to denote by $w=\sum_{j=1}^m w_j$ the total Stiefel-Whitney class of $M$.
\begin{teo}
The total Stiefel-Whitney class of the manifold $S_rM$ is
\begin{equation} \label{classtotalSM}
 w(S_rM)=\sum_{j=0}^{n}\pi^*w_j^2
\end{equation}
and in particular its mod 2 Euler class vanishes.
\end{teo}
\begin{proof}
First suppose $r$ is constant. Then the $n$-vector bundle $\kappa:=\xi^\perp\subset V$ sits in $TS_rM=H\oplus\kappa$ where we assume e.g. the Sasaki metric. We have $w(\pi^*TM)=w(H)=\pi^*w$. Clearly
$w(\pi^*TM)=w(\kappa\oplus\R\xi)=w(\kappa)$. Hence
\begin{equation*}
 w(TS_rM)=w(H\oplus\kappa)=w(\pi^*TM)^2=\pi^*w^2.
\end{equation*}
Notice $w_m(\kappa)=0$ due to rank of $\kappa$ being just $n$. Independently, $0=w_{2n+1}(S_rM)=e(S_rM)\mod 2$. Using the homeomorphism $h:S_1M\lrr S_rM,\ h(u)=ru$,  we have the result for any function $r$.
\end{proof}
\begin{Rema}
1. The results show that the odd degree Stiefel-Whitney classes of the manifolds $TM$ 
and $S_rM$ vanish.\\
2. We observe the independence of (\ref{classtotalSM}) from $r$. Moreover, always $w_1(S_rM)=0$, as expected because $TM$ is always oriented and $\xi$ induces an orientation on the submanifold.\\
3. If $M$ has a finite good cover, is oriented, and admits a non-vanishing vector
field, then we deduce $H^*(S_rM)=H^*(M)\otimes H^*(S^n)$ by the Theorem of Leray-Hirsh
(cf. \cite{BottTu}). In particular $\pi^*$ is an isomorphism $H^i(S_rM)=H^i(M)$ of
cohomology spaces up to degree $i\leq n-1=m-2$. By contrast, we have proved $\pi^*(w_m)=0$.
\end{Rema}
Since $w_2(S_rM)=w_1^2$, we have the following conclusion.
\begin{coro}
For any oriented Riemannian manifold $M$, the manifold $S_rM$ is spin.
\end{coro}
Recall $w_2$ is also the obstruction for a closed 7-manifold to admit a $\mathrm{G}_2$-structure. We have explicitly constructed a natural $\mathrm{G}_2$-structure on $S_1M$, for any oriented Riemannian 4-manifold $M$, cf.~\cite{Alb2,Alb2.1} and the references therein.

\vspace{2cm}




\begin{thebibliography}{10}


\bibitem{Abb1}
M. T. K. Abbassi,
{\em Note on the classification theorems of $g$-natural metrics on the tangent bundle of a Riemannian manifold $(M,g)$}, 
Comment. Math. Univ. Carolinae 45(4) (2004), 591--596.


\bibitem{AbbCalva}
M. T. K. Abbassi and G. Calvaruso,
{\em $g$-Natural Contact Metrics on Unit Tangent Sphere Bundles},
Monatsh. f\"ur Mathe. 151 (2006), 89--109.



\bibitem{Agri}
I. Agricola,
{\em The Srn\'\i\ lectures on non-integrable geometries with torsion},
Archi. Mathe. (Brno) Tomus 42 (2006), Suppl., 5--84.





\bibitem{Alb2}
R. Albuquerque,
{\em On the $\mathrm{G}_2$ bundle of a Riemannian 4-manifold}, 
J. Geom. Physics  60 (2010), 924--939.



\bibitem{Alb2.1}
R. Albuquerque,
{\em On the characteristic connection of gwistor space},
C. Euro. J. Math. 11(1) (2013), 149--160.



\bibitem{Alb3.0}
R. Albuquerque,
\newblock {\em Weighted metrics on tangent sphere bundles},
Q. J. Math. 63 (2) (2012), 259--273. 



\bibitem{Alb4}
R. Albuquerque,
\newblock {\em Curvatures of weighted metrics on tangent sphere bundles},
Riv. Mat. Univ. Parma Vol. 2 (2011), 299--313.



\bibitem{BottTu}
R. Bott and L. W. Tu,
{\em Differential Forms in Algebraic Topology},
Springer-Verlag Berlin, New York, 1982.


\bibitem{Dom}
P. Dombrowski,
{\em On the geometry of the tangent bundles},
J. Reine Angew. Math. 210 (1962), 73--88.


\bibitem{Geig}
H. Geiges,
{\em An Introduction to Contact Topology},
Cambridge Studies in Advanced Mathematics 109, CUP, 2008.


\bibitem{KowSek2}
O. Kowalski and M. Sekizawa,
{\em On Tangent Sphere Bundles with Small or Large Constant Radius},
Ann. Global Anal. Geom. 18 (2000), 207--219.


\bibitem{MilSta}
J. W. Milnor and J. D. Stasheff,
{\em Characteristic Classes},
Annals Math. Stud. 76, Princeton University Press (1974).




\bibitem{Sakai}
T. Sakai,
{\em Riemannian Geometry},
Translations of Mathematical Monographs 149, AMS (1996).




\bibitem{Sasa}
S. Sasaki,
{\em On the differential geometry of tangent bundles of Riemannian manifolds}, 
T\^ohoku Math. J. 10 (1958), 338--354.


\bibitem{Tash}
Y. Tashiro,
\newblock {\em On contact structures on tangent sphere bundles}, 
T\^ohoku Math. J. 21 (1969), 117--143.




\end{thebibliography}
\end{document}